\RequirePackage{fix-cm}
\documentclass[smallextended]{svjour3}       
\smartqed  
\usepackage{graphicx}
\usepackage{amsfonts,amsmath,amssymb}
\usepackage{graphicx}
\usepackage{pstricks}
\usepackage{pst-tree}
\usepackage{epsfig}
\newcommand{\ZZ}{\mathbb{Z}}
\newcommand{\ep}{\varepsilon}
\newcommand{\cI}{\mathcal I}
\newcommand{\tH}{\widetilde H}

\numberwithin{equation}{section}

\begin{document}

\title{A Hardy-Ramanujan-Rademacher-type
formula for $(r,s)$-regular partitions
\thanks{This work was partially supported by a grant from the Simons Foundation (\#209175 to James Mc Laughlin).
The second author is supported in part by National Security Agency grant H98230-11-1-0190.}}



\author{James Mc Laughlin   \and  Scott Parsell}


\institute{J. Mc Laughlin \and S. Parsell \at
Mathematics Department,
West Chester University, 25 University Avenue,
West Chester, PA 19383 \\
\email{jmclaughl@wcupa.edu, sparsell@wcupa.edu}}

\date{Received: date / Accepted: date}

\maketitle

\begin{abstract}
Let $p_{r,s}(n)$ denote the number of partitions of a positive
integer $n$ into parts containing no multiples of $r$ or $s$, where
$r>1$ and $s>1$ are square-free, relatively prime integers.
We use classical methods to derive a Hardy-Ramanujan-Rademacher-type
infinite series for $p_{r,s}(n)$.
\end{abstract}

\section{Introduction}

A \emph{partition} of a positive integer $n$ is a representation of $n$ as a sum of
positive integers, where the order of the summands does not matter.
We use $p(n)$ to denote the number of partitions of $n$, so that, for example,
$p(4)=5$, since 4 may be represented as 4, $3+1$, $2+2$, $2+1+1$ and $1+1+1+1$.
The function $p(n)$ increases rapidly with $n$, and it is difficult to compute
$p(n)$ directly for large $n$.

Rademacher \cite{Rad:37}, by slightly modifying earlier work of Hardy
and Ramanujan \cite{HR:18}, derived a remarkable infinite series for
$p(n)$. To describe this series we need some notation.
 Recall that the Dedekind sum $s(e,f)$ is defined by
\[
s(e,f):=\sum_{r=1}^{f-1}\frac{r}{f}\left(\frac{er}{f}-\left \lfloor
\frac{er}{f}\right \rfloor - \frac{1}{2}\right),
\]
and for ease of notation, we use  $\omega(e,f)$ to denote $\exp(\pi
i\,s(e,f))$, and for a positive integer $k$, set \[
A_k(n):=\sum_{\substack{0\leq h<k \\ (h,k)=1}}\omega(h,k)e^{-2\pi i n
h/k}.
\]
We recall also that $$I_{\nu}(z) = \sum_{m=0}^{\infty} \frac{(\tfrac{1}{2}z)^{\nu+2m}}{m! \Gamma(\nu+m+1)}$$
denotes the modified Bessel function of the first kind.

\begin{theorem}(Rademacher)
If $n$ is a positive integer, then
\begin{equation}\label{rad1}
p(n)=\frac{2\pi}{(24n-1)^{3/4}}\sum_{k=1}^{\infty}\frac{A_k(n)}{k}
I_{3/2}\left(\frac{\pi}{k}\sqrt{\frac{2}{3}\left (n-\frac{1}{24} \right)} \right).
\end{equation}
\end{theorem}

Rademacher's series converges incredibly fast.
For example, $$p(500)=2,300,165,032,574,323,995,027,$$ and yet six terms of the
series are sufficient to get within 0.5 of $p(500)$.  The idea of course is
that if a partial sum is known to be within $0.5$ of the value of the series, then the
nearest integer gives the exact value of $p(n)$.

Since the publication of Rademacher's paper \cite{Rad:37}, a number
of authors have found series similar to \eqref{rad1} for certain
restricted partition functions. Lehner \cite{Lehner:41} found such
series for $p_1(n)$ and $p_2(n)$, the number of partitions of $n$
into parts $\equiv \pm 1 \pmod 5$ and $\equiv \pm 2 \pmod 5$
respectively, and this was extended by Livingood \cite{L45} to
series for $p_1(n), \dots , p_{(q-1)/2}(n)$, the number of
partitions into parts $\equiv \pm 1 \pmod q$, $\equiv \pm 2 \pmod q,$ $\dots ,$
$\equiv \pm (q-1)/2 \pmod q$ respectively, where $q>3$ is an odd
prime. Hua \cite{H42} derived a Rademacher-type series for $p_O(n)$,
the number of partitions of $n$ into odd parts.

Let $q\geq 3$ be an odd prime and  $a=\{a_1,a_2,\dots ,a_r\}$ be a
set of distinct integers satisfying $1\leq a_i\leq (q-1)/2$. Hagis
\cite{Hagis:62} gave a Hardy-Ramanujan-Rademacher-type series
(H.R.R. series) for $p_a(n)$, the number of partitions of $n$ into
parts $\equiv \pm a_i \pmod q$. In a subsequent series of
papers \cite{Hagis:63,Hagis:64,Hagis:64b,Hagis:65,Hagis:66,Hagis:71,Hagis:71b},
Hagis also developed similar series for other
restricted partition functions (into odd parts, odd distinct parts,
no part repeated more than $t$ times, {\it etc.}).

Niven \cite{Niven:40} gave a H.R.R. series for $p_{2,3}(n)$, the
number of partitions of $n$ into parts containing no multiples of
$2$ or $3$. In a similar vein, Haberzetle \cite{Hab:41} gave a
series for $p_{q_1,q_2}(n)$, the number of partitions of $n$ into
parts containing no multiples of $q_1$ or $q_2$, where $q_1$ and
$q_2$ are distinct primes such that $24|(q_1-1)(q_2-1)$.

Iseki \cite{I:59,I:60,I:61} derived H.R.R. series that, amongst
other results, extended the result of Livingood \cite{L45} cited
above from a prime $q$ to a composite integer $M$, and also extended
the results of Niven  \cite{Niven:40} and Haberzetle \cite{Hab:41},
by finding a H.R.R. series for $p_M(n)$, the number of partitions of
$n$ into parts relatively prime to a square-free positive integer
$M$.

Sastri {\it et al.} \cite{PS01,S72,SV82}  derived a number of H.R.R.
series which, amongst other results, extended  the result of Hagis
cited above from a prime $q$ to an arbitrary positive integer $m$.

More recently, Sills \cite{S10a,S10b,S10c}  has partly automated
the process of finding H.R.R. series for restricted partition
functions, and aided by the use of the computer algebra system
\emph{Mathematica}, has found many new such series, including ones
for restricted partition functions represented by various identities
of Rogers-Ramanujan type.

When $r>1$ and $s>1$ are relatively prime integers, let $p_{r,s}(n)$
denote the number of partitions of $n$ into parts containing no multiples of
$r$ or $s$.  We say that such a partition of an integer $n$ is
\emph{$(r,s)$-regular}. In the present paper we give a H.R.R.
series for $p_{r,s}(n)$ when $r$ and $s$ are square-free.
We note that this result includes those
Niven \cite{Niven:40} and Haberzetle \cite{Hab:41} as special cases.

We now state our result explicitly. Define \[
F(\tau)=\frac{1}{\prod_{k=1}^{\infty}(1-e^{2\pi i k \tau})},
\]
and denote by $H_{i,j}$ a solution to the congruence $iH_{i,j}\equiv
-1 \pmod j$, and for consistency of notation below, set $H_{0,1}=0$.
For integers $k$, $r$ and $s$, let $r_k:=\gcd(r,k)$ and
$s_k:=\gcd(s,k)$ and, for ease of notation, set
\begin{equation}\label{Rdelkeq}
R:=\frac{(r-1)(s-1)}{24}, \hspace{25pt}
\delta_k:=\frac{(r/r_k-r_k)(s/s_k-s_k)}{24}.
\end{equation}
Our result may be stated as follows.

\begin{theorem}\label{t2}
Let $r>1$ and $s>1$ be square-free relatively prime integers. For a
positive integer $k$ and non-negative integer $h$ with $(h,k)=1$,
define the sequence $\{c_m(h,k)\}$ by
\begin{equation*}
\frac{ F\! \left(\! \frac{H_{h,k}}{k}+\frac{i }{z} \! \right) \!
F\! \left(\! \frac{H_{hrs/(r_ks_k),k/(r_ks_k)}}{k/(r_ks_k)}+\frac{ir_k^2s_k^2
}{rsz} \! \right)} {F\! \left(\! \frac{H_{hr/r_k,k/r_k}}{k/r_k}+\frac{ir_k^2
}{rz} \! \right) \! F\! \left(\! \frac{H_{hs/s_k,k/s_k}}{k/s_k}+\frac{i s_k^2
}{sz} \! \right)} \! := \! \sum_{m=0}^{\infty} \! c_m(h,k)\exp\biggl(\! \frac{-2\pi m
r_k s_k}{r s z} \! \biggr).
\end{equation*} If $n >R$, then
\begin{equation*}\label{t2eq}
p_{r,s}(n)=\sum_{k=1}^{\infty}\sum_{m=0}^{\lfloor \delta_k\rfloor}
\! \frac{2\pi A_{k,m}(n)}{k} \! \sqrt{  \frac{r_ks_k(\delta_k-m)}{rs(n-R)}
} I_1\biggl(\frac{4\pi}{k}\! \sqrt{\frac{r_ks_k}{rs}(\delta_k-m)(n-R) }\biggr),
\end{equation*}
where
\begin{equation*}
A_{k,m}(n):=\sum_{\substack{h=0 \\ (h,k)=1}}^{k-1}
\frac{\omega(h,k)\omega(hrs/(r_ks_k),k/(r_ks_k))}
{\omega(hr/r_k,k/r_k)\omega(hs/s_k,k/s_k)}
c_m(h,k)\exp\left(\frac{-2\pi i n h}{k} \right).
\end{equation*}
\end{theorem}

The method of proof follows to a large extent the method used by
previous authors to derive similar convergent series for other
partition functions. In section \ref{it}, the Cauchy Residue Theorem
is applied to the generating function for the sequence $p_{r,s}(n)$, and
a change of variable is then applied to convert the path of
integration to the line segment $[i,i+1]$. Next, this line segment
is deformed to follow the path along the top of a collection of Ford
circles, after which another change of variable transforms the arc
along the top of each Ford circle to an arc along the circle in the
complex plain with center 1/2 and radius 1/2. Next, the
transformation formula for the Dedekind eta function $\eta(\tau)$ is
used to transform the integrand into a form whose properties can be
exploited to derive the final series stated in Theorem \ref{t2}.
Each transformed infinite product is expanded in a series, which is
broken into an initial finite part (which eventually leads to the
series of the theorem) and a tail, whose contribution is shown to be
negligible.

The path of integration for each of the terms coming from the tail
of the series mentioned above is divided into three arcs.  In
section \ref{ks}, Kloosterman sum estimates are developed, which are
used in section \ref{et} to get error bounds on the integrals along
the three arcs for each term in the tail.  This shows that these error
terms go to zero as $N\to \infty$, where $N$ is the order of the
Farey sequence giving rise to the collection of Ford circles.

In section \ref{mt}, the arcs of integration along the circle with
center 1/2 and radius 1/2 for the main terms are replaced with a new
path along the entire circle. It is shown that the contributions
from the additional arcs also go to zero as
$N\to \infty$, where $N$ is as in the paragraph above. Two other
changes of variable and an application of an integral formula for
modified Bessel functions of the first kind lead the final result.

Remark: With the notation for $F(\tau)$ as above and for
$\eta(\tau)$ as below, the generating functions
\[
e^{-\pi i(r-1)(s-1))\tau/12} \frac{F(\tau)F(r s
\tau)}{F(r\tau)F(s\tau)}=\frac{\eta(r\tau)\eta(s\tau)}{\eta(\tau)\eta(r
s \tau)}
\]
are weight-zero modular forms, so that the general theorem of
Bringmann and Ono \cite{BO11} could in theory be used to derive our
series for $p_{r,s}(n)$. However, we prefer to employ the
Hardy-Ramanujan-Rademacher method.

\section{Initial transformations}\label{it}

Write $(q;q)_{\infty} = \prod_{j=1}^{\infty} (1-q^j)$, and let
\begin{equation}\label{Geq}
G(x)=\sum_{n=0}^{\infty}p_{r,s}(n)x^n
=\frac{(x^r;x^r)_{\infty}(x^s;x^s)_{\infty}}{(x;x)_{\infty}(x^{rs};x^{rs})_{\infty}}
\end{equation}
denote the generating function for the sequence $\{p_{r,s}(n)\}$. By
the Cauchy Residue Theorem,
\[
p_{r,s}(n)=\frac{1}{2\pi i}\int_C \frac{G(x)}{x^{n+1}}dx,
\]
where $C$ is any positively oriented simple closed curve inside the
unit circle containing the origin. As usual, we start by taking $C$
to be the circle centered at the origin with radius $e^{-2\pi}$, and
make the change of variable $x=e^{2\pi i \tau}$ to get
\[
p_{r,s}(n)=\int_{i}^{i+1}G(e^{2\pi i \tau})e^{-2\pi i n \tau}d\tau.
\]
We follow Rademacher by deforming the path of integration so that it
traces the upper arcs of the collection of Ford circles
\[
\left \{C_{h,k}:\frac{h}{k} \in \mathcal{F}_N\right \},
\]
where $C_{h,k}$ is the circle with center $h/k+i/(2k^2)$ and radius
$1/(2k^2)$, and $\mathcal{F}_N$ is the set of Farey fractions of
order $N$. We denote  the part of the path that is an arc of the
circle $C_{h,k}$ by $\gamma(h,k)$. Thus
\begin{equation}\label{prseq2}
p_{r,s}(n)=\sum_{k=1}^N\sum_{\substack{h=0 \\ (h,k)=1}}^{k-1}\int_{\gamma(h,k)}\frac{
F(\tau)F(r s \tau)e^{-2\pi i n \tau}}{F(r\tau)F(s\tau)}d\tau.
\end{equation}
Next, for each circle $C_{h,k}$, set $z=-ik^2(\tau-h/k)$,
transforming the circle $C_{h,k}$ to the circle $K$ with center $1/2$
and radius $1/2$, and transforming the arc $\gamma(h,k)$ to the arc
(not passing through $0$) on the latter circle joining the points
\[
z_1(h,k)=\frac{k^2+ikk_1}{k^2+k_1^2} \qquad \mbox{and} \qquad z_2(h,k)=\frac{k^2-ikk_2}{k^2+k_2^2},
\]
where $h_1/k_1<h/k<h_2/k_2$ are consecutive Farey fractions in
$\mathcal{F}_N$. With these changes,
\begin{equation}\label{prseq3}
p_{r,s}(n)= \sum_{k=1}^N\sum_{\substack{h=0 \\ (h,k)=1}}^{k-1}
\int_{z_1(h,k)}^{z_2(h,k)}\frac{
F\left(\frac{h}{k}+\frac{iz}{k^2}\right)F\left(r s
\left(\frac{h}{k}+\frac{iz}{k^2}\right)\right)e^{-2\pi i n
\left(\frac{h}{k}+\frac{iz}{k^2}\right)}\,i}
{F\left(r\left(\frac{h}{k}+\frac{iz}{k^2}\right)\right)
F\left(s\left(\frac{h}{k}+\frac{iz}{k^2}\right)\right)k^2}dz.
\end{equation}

Next, recall that the Dedekind eta function is defined by
\[
\eta(\tau) =e^{\pi i \tau/12}\prod_{k=1}^{\infty}(1-e^{2\pi i k
\tau})
\]
and satisfies the transformation formula (see for example Apostol \cite{Apostol:MF}, Theorem 3.4)
\begin{equation}\label{etatrans}
\eta\left(\frac{a\tau+b}{c \tau+d}\right)=\exp\left(\pi
i\left(\frac{a+d}{12c}+s(-d,c) \right)\right)\{-i(c\tau
+d)\}^{1/2}\eta(\tau)
\end{equation}
whenever $\biggl(\!  \begin{array}{c c} a & b \\ c & d \end{array}\!  \biggr)$ is an element of the modular group,
$c > 0$, and $\tau$ lies in the upper half-plane .  Thus
\begin{multline*}
F(\tau)=
\exp\left(\frac{\pi i}{12}\left(\tau - \frac{a\tau+b}{c
\tau+d}+\frac{a+d}{c}\right) \right) \\ \times \exp\left(\pi i s(-d,c) \right)
\{-i(c\tau +d)\}^{1/2}F\left(\frac{a\tau+b}{c \tau+d}\right).
\end{multline*}

In what follows, for each set of choices for $a$, $c$ and $d$, we
take $b$ to be $(a d - 1)/c$. For $v\in\{1,r,s,rs\}$,
$(k,rs)=r_ks_k$ and $\tau=h/k+i z/k^2$, we set $v_k=(v,k)$, so that
$(v,v_k)\in\{(1,1),(r,r_k),(s,s_k),(rs,r_ks_k)\}$.  We then transform
$F(v\tau)$ by setting $c=k/v_k$, $d=-hv/v_k$ and
$a=H_{hv/v_k,k/v_k}$, to get
\begin{multline}\label{Ftransz}
F\left(\frac{v h}{k}+\frac{i v z}{k^2}\right)= \exp\left(\frac{\pi
v_k^2}{12vz}-\frac{\pi v z}{12k^2} \right)\exp\left(\pi i\,
s\left(\frac{hv}{v_k},\frac{k}{v_k}\right) \right)\\
\times\left\{\frac{vz}{kv_k}\right\}^{1/2}F\left(\frac{H_{hv/v_k,k/v_k}}{k/v_k}+\frac{i
v_k^2}{vz} \right).
\end{multline}
On substituting into (\ref{prseq3}), this gives
\begin{multline}\label{prseq4}
p_{r,s}(n)=  \sum_{k=1}^N \sum_{\substack{h=0 \\ (h,k)=1}}^{k-1}
\frac{\omega(h,k)\omega(hrs/(r_ks_k),k/(r_ks_k))}
{\omega(hr/r_k,k/r_k)\omega(hs/s_k,k/s_k)} \exp\left(\frac{-2\pi i n
h}{k} \right)\frac{i}{k^2}
\\
\times \int_{z_1(h,k)}^{z_2(h,k)}\!\!
\exp\left(2\pi \left(\frac{r_ks_k\delta_k}{rsz}+\frac{(n-R)z}{k^2}
\right) \right)\\
\times\frac{ F\left(\frac{H_{h,k}}{k}+\frac{i }{z} \right)
F\left(\frac{H_{hrs/(r_ks_k),k/(r_ks_k)}}{k/(r_ks_k)}+\frac{ir_k^2s_k^2
}{rsz} \right)} {F\left(\frac{H_{hr/r_k,k/r_k}}{k/r_k}+\frac{ir_k^2
}{rz} \right) F\left(\frac{H_{hs/s_k,k/s_k}}{k/s_k}+\frac{i s_k^2
}{sz} \right)}dz.
\end{multline}

We temporarily fix $v \in \{1,r,s,rs\}$ and introduce the shorthand $g=v_k=(v,k)$.  We observe that
the congruences
$$H_{vh/g,k/g} (vh/g) \equiv -1 \pmod{k/g} \qquad \mbox{and} \qquad H_{h,k} h \equiv -1 \pmod{k}$$
imply that
\begin{equation*}
\label{HHrel} vH_{vh/g,k/g} \equiv g H_{h,k} \pmod{k}
\end{equation*}
when $(h,k)=1$.  Since $r$ and $s$ are square-free, we have
$(v/g,k)=1$, and hence the congruence
$$h(v/g) \tH_{h,k} \equiv -1 \pmod{k}$$
has a solution $\tH_{h,k}$, and we are free to take $H_{h,k} =
(v/g)\tH_{h,k}$ to be a multiple of $v/g$.  In particular, then, one
has $vH_{vh/g,k/g} \equiv gH_{h,k} \pmod{v}$, and since $(v/g,k)=1$
it follows from (\ref{HHrel}) and the Chinese Remainder Theorem that
\begin{equation*}
\label{HHrel2} vH_{vh/g,k/g} \equiv gH_{h,k} \pmod{vk/g}.
\end{equation*}
Hence the periodicity of $F(\tau)$ implies that
\begin{equation*}
\label{Ftrans2}
F\left(\frac{gH_{vh/g,k/g}}{k}+\frac{ig^2}{vz}\right) =
F\left(\frac{g^2 H_{h,k}}{vk}+\frac{ig^2}{vz}\right).
\end{equation*}

Put
$$\mu_z=\frac{r_ks_k}{rs}\left(\frac{H_{h,k}}{k}+\frac{i}{z}\right).$$
Then we deduce from (\ref{Ftrans2}) that the ratio appearing in
(\ref{prseq4}) is
\begin{equation}
\label{Fratio}
\frac{F((rs/r_ks_k)\mu_z)F(r_ks_k
\mu_z)}{F((r_ks/s_k)\mu_z)F((s_kr/r_k)\mu_z)} := G^*(\mu_z),
\end{equation}
where we write
\begin{equation}
\label{G*def} G^*(\tau) = \sum_{m=0}^{\infty} c_{m,k} \exp(2\pi i m
\tau)
\end{equation}
for some coefficients $c_{m,k}$.  We note that the coefficients
$c_m(h,k)$ occurring in the statement of Theorem \ref{t2} satisfy
$$c_m(h,k)=c_{m,k} \exp\biggl(\frac{2\pi i m r_ks_k H_{h,k}}{rsk}\biggr),$$
so that in particular $|c_m(h,k)|=|c_{m,k}|$.
Then (\ref{prseq4}) may be expressed as
\begin{multline*}
p_{r,s}(n)=\sum_{\substack{h,k \\ (h,k)=1}} \frac{i}{k^2}
\Omega_{h,k} e^{-2\pi i n h/k} \\ \times \int_{z_1(h,k)}^{z_2(h,k)}
\exp\left(2\pi\biggl(\frac{r_ks_k\delta_k}{rsz}+
\frac{(n-R)z}{k^2}\biggr)\right) G^*(\mu_z) dz,
\end{multline*}
where
\begin{equation}
\label{Omegadef} \Omega_{h,k} = \frac{\omega(h,k)
\omega(rsh/(r_ks_k),k/(r_ks_k))}{\omega(rh/r_k,k/r_k)\omega(sh/s_k,k/s_k)}.
\end{equation}
We further introduce the notation
\begin{equation}
\label{Psidef} \Psi_{m,k}(z)=\exp\left(\frac{2\pi
r_ks_k(\delta_k-m)}{rsz}+\frac{2\pi(n-R) z}{k^2}\right),
\end{equation}
which allows us to write
\begin{multline}
\label{prsformula} p_{r,s}(n) = \sum_{k=1}^N \frac{i}{k^2}
\sum_{m=0}^{\infty} c_{m,k} \sum_{\substack{0 \leq h \leq k-1 \\
(h,k)=1}} \Omega_{h,k} \exp\left(\frac{2\pi
i}{rsk}(r_ks_kmH_{h,k}-rsnh)\right) \\  \times \int_{z_1(h,k)}^{z_2(h,k)}
\Psi_{m,k}(z) \, dz.
\end{multline}

We decompose the sum over $m$ into two parts, $m < \delta_k$ and $m \geq \delta_k$, and write
\begin{equation}
\label{prsdecomp} p_{r,s}(n) = P_1(n; N) + P_2(n; N)
\end{equation}
for the resulting decomposition of (\ref{prsformula}).  We aim to
show that $P_2(n; N)$ contributes a negligible amount to the
formula. We find it useful to split the path of integration from
$z_1(h,k)$ to $z_2(h,k)$ into the three arcs $[z(k_1),z(N)]$,
$[z(-N),z(-k_2)]$, and $[z(-N),z(N)]$, and we further decompose the
first two as unions of arcs of the shape $[z(l), z(l+1)]$, where
\begin{equation}
\label{zldef} z(l) = \frac{k^2}{k^2+l^2} + \frac{ikl}{k^2+l^2}.
\end{equation}
It is easy to check that each $z(l)$ lies on the circle
$|z-1/2|=1/2$. Since $k_1, k, k_2$ are denominators of consecutive
elements in the Farey sequence of order $N$, we have $k+k_1 \geq N+1$
and $k+k_2 \geq N+1$, and hence $k_1 \geq N+1-k$ and $-k_2 \leq k-N-1$
for all values of $h$ and $k$.  Moreover, since $hk_1-h_1k =
h_2k-hk_2 = 1$, we have $hk_1 \equiv 1 \pmod{k}$ and $hk_2 \equiv -1
\pmod{k}$.  It follows that $H_{h,k} \equiv k_2 \equiv -k_1
\pmod{k}$, and hence the condition $-k_2 \leq l \leq k_1-1$ is
equivalent to a restriction of $H_{h,k}$ to some interval $\cI_l$
modulo $k$. We may therefore interchange the order of summation and
integration in (\ref{prsformula}) to obtain
\begin{equation}
\label{Ebd1} P_2(n; N) = \sum_{k=1}^N \frac{i}{k^2} \sum_{m=\lceil
\delta_k \rceil}^{\infty} c_{m,k} (S_1(m,k)+S_2(m,k)+S_3(m,k)),
\end{equation}
where
$$S_1=\sum_{l=N+1-k}^{N-1} \int_{z(l)}^{z(l+1)} \Psi_{m,k}(z)
\Theta(k,l,m)\,dz,$$
$$S_2 = \sum_{l=-N}^{k-N-2}
\int_{z(l)}^{z(l+1)}\Psi_{m,k}(z)\Theta(k,l,m) \, dz,$$
$$S_3 =  \int_{z(-N)}^{z(N)} \Psi_{m,k}(z) \Theta(k,l,m) \, dz,$$
and
$$\Theta(k,l,m)=\sum_{\substack{0 \leq h \leq k-1 \\ (h,k)=1 \\ H_{h,k} \in \cI_l}}
 \Omega_{h,k} \exp\left(\frac{2\pi i}{rsk}(r_ks_kmH_{h,k}-rsnh)\right).$$
In order to make further progress, we must develop suitable estimates for $\Theta(k,l,m)$.
We take up this task in the next section.

\section{Kloosterman sums}\label{ks}

In order to estimate $S_1$, $S_2$, and $S_3$, we aim to express
$\Theta(k,l,m)$ in terms of Kloosterman sums. As a first step, we
are able to remove the restriction on $H_{h,k}$ in the summation at
a cost of $O(\log k)$. The argument is similar to that of Hagis
\cite{Hagis:71} (see also Lehner \cite{Lehner:41}).

\begin{lemma} For each $k$ and $m$, there exists an integer $j=j(k,m)$ with $0 \leq j \leq k-1$ such that for every $l$ one has
\label{Thetabd1}
\begin{equation*}
|\Theta(k,l,m)| \ll (1+\log k) \left|\sum_{\substack{h=0 \\(h,k)=1}}^{k-1} \Omega_{h,k} \exp\biggl(\frac{2\pi
i}{rsk}((r_ks_km+rsj)H_{h,k}-rsnh)\! \biggr) \right|,
\end{equation*}
where the implicit constant is absolute.
\end{lemma}

\begin{proof}
Fix $k$, $l$, and $m$, and let $\cI_l = [\alpha,\beta]$, where
$\alpha$ and $\beta$ are integers with $0 \leq \beta -\alpha < k$.
By orthogonality, we have
$$\frac{1}{k} \sum_{j=0}^{k-1} \exp\left(\frac{2\pi i j (H-t)}{k}\right) = \begin{cases} 1  & {\rm if} \ \ H \equiv t \! \pmod{k} \\ 0  & {\rm if} \ \ H \not \equiv t \! \pmod{k} \end{cases},$$
and hence the expression
$$\frac{1}{k} \sum_{t=\alpha}^{\beta} \sum_{j=0}^{k-1} \exp\left(\frac{2\pi i j (H-t)}{k}\right)$$
is 1 if $H$ is congruent mod $k$ to one of the integers in $\cI_l$
and 0 otherwise.  We therefore have
\begin{equation}
\label{Theta} \Theta(k,l,m) = \frac{1}{k} \sum_{j=0}^{k-1} \gamma_j
\sum_{\substack{0 \leq h \leq k-1 \\ (h,k)=1}} \Omega_{h,k}
\exp\left(\frac{2\pi i}{rsk}((r_ks_km+rsj)H_{h,k}-rsnh)\right),
\end{equation}
where
$$\gamma_j = \sum_{t=\alpha}^{\beta} \exp\left(\frac{-2\pi i j t}{k}\right).$$
By summing this geometric progression, we find that
$$|\gamma_j| \leq \min\left(\beta-\alpha, \frac{1}{\sin(\pi j /k)}\right) \leq \min(k, \tfrac{1}{2}||j/k||^{-1}),$$
where $||\cdot||$ denotes the distance to the nearest integer.  One
now easily gets (see for example Lemma 3.2 of Baker \cite{Bak:DI})
$$\sum_{j=0}^{k-1} |\gamma_j| \ll k(1+\log k),$$
and the lemma follows after taking the maximum over $j$ in the inner
summation of (\ref{Theta}).
\end{proof}

Write $24rs = AB$, where $A$ is the largest divisor of $24rs$
relatively prime to $k$, and let $\bar{A}$ denote the multiplicative
inverse of $A$ modulo $Bk$.    We note that every prime factor of
$B$ is a prime factor of $k$, whence $\gcd(h,k)=1$ if and only if
$\gcd(h,Bk)=1$.  Moreover, for each such $h$ and $k$ we can find
$H_{h,Bk}$ with the property that $hH_{h,Bk} \equiv -1 \pmod{Bk}$
and $A|H_{h,Bk}$.  These observations allow us to calculate the
$\Omega_{h,k}$ defined by (\ref{Omegadef}) rather explicitly.

\begin{lemma}
\label{Omegacalc} Suppose that $(h,k)=1$, let $A$, $B$, and
$H_{h,Bk}$ be as above, and additionally write
$\nu_k=(r_k-1)(s_k-1)$,  $\sigma_k=(r-r_k)(s-s_k)$, and
$$\Phi_{h,k} = \exp\left(\frac{48\pi i\bar{A}}{Bk}(k^2R-r_ks_k\delta_k)H_{h,Bk}\right),$$
where $R$ and $\delta_k$ are as in $(\ref{Rdelkeq})$. When $k$ is
odd one has
$$\Omega_{h,k} = \left(\! \frac{r/r_k}{s_k} \! \right) \! \left(\! \frac{s/s_k}{r_k} \! \right) \exp\left(\frac{-\pi i k \nu_k}{4r_ks_k}\right) \exp\left(\frac{-2\pi i}{kr_ks_k} (k^2\delta_k-r_ks_kR) h \right) \Phi_{h,k},$$
and when $k$ is even one has
$$\Omega_{h,k} = \left(\! \frac{s_k}{r/r_k} \! \right) \! \left(\! \frac{r_k}{s/s_k} \! \right) \exp\left(\frac{2\pi i}{kr_ks_k}(2k^2\delta_k+\tfrac{1}{8}k\sigma_k+r_ks_kR)h  \right) \Phi_{h,k}.$$
\end{lemma}

\begin{proof}
When $v \in \{1,r,s,rs\}$, write $g=(v,k)$ and
\begin{equation}
\label{omegav} \omega_v(h,k) = \exp\left(\frac{-2\pi
i(k^2-g^2)}{24kg^3}\bigl(2hvg+(v^2h^2-g^2)H_{vh/g,k/g}\bigr)\right).
\end{equation}
Then when $k$ is odd, formula (2.4) of Niven \cite{Niven:40} gives
\begin{equation}
\label{kodd} \omega(vh/g,k/g) =
\left(\frac{-hv/g}{k/g}\right)\exp\left(\frac{-\pi i}{4}(k-1)\right)
\omega_v(h,k).
\end{equation}
When $k$ is even, the condition that $v$ is square-free implies that
$hv/g$ is odd, and hence we may apply formula (2.3) of
\cite{Niven:40} to obtain
\begin{equation}
\label{hodd} \omega(vh/g,k/g) =
\left(\frac{-k/g}{hv/g}\right)\exp\left(\frac{-\pi
i}{4}\biggl(2-\frac{hv}{g^2}(k+g)\biggr)\right) \omega_v(h,k).
\end{equation}
Since $v$ is square-free, we have $(vh/g,Bk)=1$.  Therefore, as in
the argument preceding the statement of the lemma, we may replace
each $H_{vh/g,k/g}$ in (\ref{omegav}) by an integer $H_{hv/g,Bk/g}$
divisible by $A$, and the argument leading to (\ref{HHrel2}) then
gives
$$H_{hv/g,Bk/g} \equiv \frac{g}{v} H_{h,Bk} \pmod{Bk/g},$$
where we recall that $H_{h,Bk}$ is divisible by $A$ and hence by
$v/g$.  Substituting into (\ref{omegav}) now gives
\begin{align*}
\omega_v(h,k) &= \exp\left(\frac{-\pi i}{6kg^2}(k^2-g^2)vh\right)
\exp\left(\frac{-2\pi i \bar{A}rs}{Bkvg^2}(k^2-g^2)
(vh^2-g^2)H_{h,Bk} \! \right) \\ &=
\exp\left(\frac{-\pi i} {12kg^2}(k^2-g^2)vh\right)
\exp\left(\frac{2\pi i \bar{A}rs} {Bkv}(k^2-g^2)H_{h,Bk} \!  \right),
\end{align*}
upon noting that $v|rs$, $g^2|(k^2-g^2)$, and $h^2H_{h,Bk} \equiv -h
\pmod{Bk}$. It now follows with a bit of computation that
$$\frac{\omega_1(h,k)\omega_{rs}(h,k)}{\omega_r(h,k)\omega_s(h,k)} = \exp\left(\frac{-2\pi i}{k}\biggl(R-\frac{k^2\delta_k}{r_ks_k}\biggr)hv\right) \Phi_{h,k}.$$
The lemma now follows from (\ref{kodd}) and (\ref{hodd}) via routine
calculations using the multiplicative properties of the Jacobi
symbol.
\end{proof}

We now show that the summation on the right hand side of Lemma
\ref{Thetabd1} is a Kloosterman sum with modulus $Bk$.  Fix $j=j(k,m)$ to
be the integer in the statement of Lemma \ref{Thetabd1} for which
the expression on the right is maximal and write $T(k,m)$ for the
corresponding sum, so that for each $l$ one has
\begin{equation}
\label{thetaT} \Theta(k,l,m) \ll (1+\log k) T(k,m).
\end{equation}
From the definition of the Dedekind sum (see Section 1), together with (\ref{Omegadef}), we see that
$\Omega_{h+tk, k} = \Omega_{h,k}$ for all $t \in \ZZ$.  Hence we
can write
\begin{equation*}
\label{T1}
 T(k,m) = B^{-1} \sum_{\substack{0 \leq h \leq Bk-1 \\ (h,Bk)=1}} \Omega_{h,k} \exp\left(\frac{2\pi i}{Bk}\bigl((24\bar{A}r_ks_km+jB)H_{h,Bk}-Bnh\bigr)\right),
\end{equation*}
since the definition of $B$ implies that $(h,k)=1$ if and only if
$(h,Bk)=1$. Moreover, since $Ah$ runs over a reduced residue system
modulo $Bk$ as $h$ does and since $-h^{-1} \equiv AH_{Ah,Bk} \equiv
H_{h,Bk} \pmod{Bk}$, we find that
\begin{equation}
\label{T2}
 T(k,m) = B^{-1} \! \! \! \! \! \sum_{\substack{0 \leq h \leq Bk-1 \\ (h,Bk)=1}} \! \! \! \! \! \Omega_{Ah,k} \exp\left(\frac{2\pi i}{Bk}\bigl((24\bar{A}r_ks_km+jB)\bar{A}H_{h,Bk}-ABnh\bigr)\right).
\end{equation}
We are now able to express $T(k,m)$ in terms of the Kloosterman sum
$$K(a, b; c) = \sum_{\substack{1 \leq x \leq c \\ (x,c)=1}} \exp\left(\frac{2\pi i (ax+b\bar{x})}{c}\right),$$
where $\bar{x}x \equiv 1 \pmod{c}$.  In our case $c=Bk$, and
$-H_{h,Bk}$ plays the role of $\bar{x}$.  According to Weil's bound
(see for example Iwaniec and Kowalski \cite{IK:ANT}, Corollary
11.12) one has
\begin{equation}
\label{Kbound} K(a, b; c) \ll (a,b,c)^{1/2} c^{1/2+\ep},
\end{equation}
and this delivers the bound on $\Theta(k,l,m)$ recorded in the
following lemma.

\begin{lemma}
\label{Thetabd2} One has $\Theta(k, l, m) \ll k^{1/2+\ep}$, where
the implicit constant depends at most on $\ep$, $r$, $s$, and $n$.
\end{lemma}
\begin{proof}
On substituting the results of Lemma \ref{Omegacalc} (with $h$
replaced by $Ah$) into (\ref{T2}), we obtain
$$|T(m,k)|=B^{-1} |K(a,b;Bk)|$$ where $$b=24\bar{A}^2(r_ks_k(\delta_k-m)-k^2R)-j\bar{A}B,$$
and $$a=rs\left(\frac{k\alpha_k}{r_ks_k}+24R\right)-ABn,$$ where
$\alpha_k = -24k^2\delta_k$ if $k$ is odd and
$\alpha_k=48k^2\delta_k+3k\sigma_k$ if $k$ is even.
Since $r_ks_k|k$ and $k|\alpha_k$, any common divisor of $a$ and
$Bk$ must also divide the integer
$$u=24\left(ABn-24rsR\right) = 576rs(n-R).$$
In view of the hypothesis that $n > R$, we have $u \neq 0$ and hence
$(a,b,Bk) \ll_{r,s,n} 1$.  The lemma now follows from (\ref{thetaT})
and (\ref{Kbound}).
\end{proof}

\section{The error terms}\label{et}

In order to complete the analysis of $P_2(n; N)$, we require an estimate for the growth rate of the coefficients $c_{m,k}$ in (\ref{prsformula}) arising from the expansion (\ref{G*def}).  The following crude bound
will suffice for our purposes.
\begin{lemma}
\label{cmkbound} One has
\begin{equation*}
c_{m,k}  \ll e^{2\pi \sqrt{m}},
\end{equation*}
where the implicit constant is independent of $k$.
\end{lemma}
\begin{proof}
For simplicity, we consider the series $$g(x) = \sum_{m=0}^{\infty}
c_{m,k} x^m,$$ so that (\ref{G*def}) gives $G^*(\tau)=g(e^{2\pi i \tau})$. Then by (\ref{Fratio}) one has
$$g(x) =
\frac{(x^a;x^a)_{\infty}(x^b;x^b)_{\infty}}{(x^c;x^c)_{\infty}(x^d;x^d)_{\infty}},$$
where $a=r_ks/s_k$, $b=s_kr/r_k$, $c=rs/r_ks_k$, and $d=r_ks_k$.  We
have $$\frac{1}{(x^t;x^t)_{\infty}} = \sum_{l=0}^{\infty}
p(l)x^{tl},$$ and it follows that the coefficient of $x^m$ in
$[(x^c;x^c)_{\infty}(x^d;x^d)_{\infty}]^{-1}$ is bounded above by
$(m+1)p(m)^2$. Furthermore, by Euler's Pentagonal number theorem we
have
$$(x^t;x^t)_{\infty} = \sum_{l=-\infty}^{\infty} (-1)^l x^{tl(3l-1)/2},$$
and from this one sees that the coefficient of $x^m$ in
$(x^a;x^a)_{\infty}(x^b;x^b)_{\infty}$ has absolute value at most
$4\sqrt{m}+2$. Hence on applying the well-known Hardy-Ramanujan asymptotic formula \cite{HR:18} for $p(m)$,
we deduce that
$$c_{m,k} \ll m^{5/2} p(m)^2 \ll m^{1/2} e^{2\pi(2/3)^{1/2} \sqrt{m}} \ll e^{2\pi \sqrt{m}},$$
and the lemma follows.
\end{proof}

We are now able to show that the terms in (\ref{prsformula}) with $m
\geq \delta_k$ contribute a negligible amount.  First of all, it
follows from Lemma \ref{Thetabd2} and the definitions at the end of Section \ref{it} that
$$S_1 \ll k^{1/2+\ep} \int_{z(N+1-k)}^{z(N)} |\Psi_{m,k}(z)| \, dz, \qquad S_2 \ll k^{1/2+\ep} \int_{z(-N)}^{z(k-N-1)} |\Psi_{m,k}(z)| \, dz,$$ and
$$S_3 \ll k^{1/2+\ep} \int_{z(-N)}^{z(N)} |\Psi_{m,k}(z)| \, dz.$$
Since ${\rm Re}(z) \leq 1$ in $|z-1/2| \leq 1/2$, the definition (\ref{Psidef}) immediately gives
\begin{equation}
\label{Psibound} \Psi_{m,k}(z) \ll_{n} \exp\biggl(\frac{2\pi
r_ks_k}{rsz}\left(\delta_k-m\right)\biggr).
\end{equation}
Moreover, one has ${\rm{Re}}(1/z) \geq 1$ in the disk $|z-1/2| \leq 1/2$ and
it follows that $\Psi_{m,k}(z) \ll 1$ for all $m \geq \delta_k$. If
$\delta_k < 0$ then (\ref{Psibound}) yields $\Psi_{m,k}(z) \ll
e^{-2\pi m/(rs)}$, whereas if $\delta_k > 0$ and $m > 2\delta_k$
then we obtain $\Psi_{m,k}(z) \ll e^{-\pi m/(rs)}$.

With the above estimates in hand, it remains to bound the lengths of the various arcs of
integration.  After recalling (\ref{zldef}),  a simple calculation
reveals that
$$|z(N)-z(N-k+1)| \ll k^2/N^2 \quad \mbox{and} \quad |z(k-N-1)-z(-N)| \ll k^2/N^2,$$
while $|z(N)-z(-N)| \ll 1/N$.  Therefore, on shifting the paths of
integration from the circle to the respective chords connecting the
endpoints, we deduce from (\ref{Ebd1}), Lemma \ref{cmkbound}, and
the discussion following (\ref{Psibound}) that
\begin{align}
\label{Ebd2}
P_2(n; N) &\ll \sum_{k=1}^N (N^{-2} k^{1/2+\ep} + N^{-1} k^{-3/2+\ep})\left(1+\sum_{m > 2\delta_k} |c_{m,k}| e^{-\pi m/(rs)}\right) \notag \\
&\ll N^{-1/2+\ep}.
\end{align}
Thus on recalling (\ref{prsdecomp}) we  get
\begin{equation}
\label{P2neg} p_{r,s}(n) = P_1(n; N) + O(N^{-1/2+\ep}),
\end{equation}
and hence it suffices to analyze $P_1(n; N)$.

\section{The main terms}\label{mt}

For each $k$, we now consider the main terms (if any) with $0 \leq m
< \delta_k$.  Recall that $K$ is the circle with center $1/2$ and radius $1/2$,
and let $K(-)$ denote this circle traversed in the clockwise
direction. We write
\[
\int_{z_1(h,k)}^{z_2(h,k)}
=\int_{K(-)}-\int_{0}^{z_1(h,k)}-\int_{z_2(h,k)}^{0},
\]
and use this to decompose each of the integrals in \eqref{prsformula}.
Our aim is to show that the integrals over the arcs
$[z_1(h,k),z_2(h,k)]$ in (\ref{prsformula}) can be replaced by
integration over $K(-)$,
with negligible error.  By repeating the argument leading to
(\ref{Ebd1}), we find that the contribution from
$\int_{0}^{z_1(h,k)}$ and $\int_{z_2(h,k)}^0$ is at most
\begin{equation}
\label{Ebd3} P_3(n; N) = \sum_{k=1}^N \frac{1}{k^2} \sum_{0 \le m
< \delta_k} (|S_1(k,m)| + |S_2(k,m)| + |S_3(k,m)|).
\end{equation}
Since the coefficient of $1/z$ in the exponent of $\Psi_{m,k}(z)$ is
positive when $m < \delta_k$, we keep the path of integration on the
circle, where we have ${\rm Re}(1/z)=1$, and hence (\ref{Psibound})
gives $\Psi_{m,k}(z) \ll 1$.  Finally, it is easy to show (see for
example the proof of Apostol \cite{Apostol:MF}, Theorem 5.9) that
each of the arcs $[z(N-k+1), z(N)]$, $[z(-N), z(k-N-1)]$, and
$[z(-N), z(N)]$ has length $O(k/N)$.  It therefore follows from
(\ref{Ebd3}) and Lemma \ref{Thetabd2} that
\begin{equation}
\label{Ebd4} P_3(n; N) \ll N^{-1} \sum_{k=1}^N k^{-1/2+\ep} \ll
N^{-1/2+\ep}.
\end{equation}
On letting $N \to \infty$, we deduce from (\ref{P2neg}) and
(\ref{Ebd4}) that
\begin{equation*}\label{prseq5}
p_{r,s}(n)=
\sum_{k=1}^{\infty}\sum_{m =0}^{\lfloor \delta_k \rfloor} \!
A_{k,m}(n)\frac{i}{k^2}
 \int_{K(-)} \! \! \! \!
\exp\left(\frac{2\pi r_ks_k}{rsz} \left(\delta_k-m \right)+
\frac{2\pi z}{k^2}\left(n-R \right)\right)dz,
\end{equation*}
where $A_{k,m}(n)$ is as in the statement of Theorem \ref{t2}.
It remains to express the integral over $K(-)$ in terms of modified Bessel
functions of the first kind.  Setting $w=1/z$ gives
\begin{multline*}
p_{r,s}(n) = \sum_{k=1}^{\infty}\sum_{0 \leq m < \delta_k}
A_{k,m}(n) \frac{i}{k^2} \\
\times \int_{1-i\infty}^{1+i\infty}\ \exp\left(\frac{2\pi w r_ks_k}{rs}
\left(\delta_k-m \right)+ \frac{2\pi
}{k^2w}\left(n-R\right)\right)\frac{-1}{w^2}dw.
\end{multline*}
We now set
\[
t=\frac{2\pi wr_ks_k}{rs} \left(\delta_k-m \right) \qquad \mbox{and} \qquad
c=\frac{2\pi r_ks_k}{rs} \left(\delta_k-m\right)
\]
to get
\begin{multline}\label{prseq55}
p_{r,s}(n)= \sum_{ k=1}^{\infty}\sum_{0\leq m <\delta_k}
A_{k,m}(n)\frac{2\pi c}{k^2}
\\
\times\frac{1}{2\pi i}\int_{c-i\infty}^{c+i\infty}t^{-2}\exp\left(t+
\frac{4\pi^2r_ks_k}{k^2rs} \left(\delta_k-m
\right)\left(n-R\right)\frac{1}{t}\right)dt.
\end{multline}
Lastly, we use the formula
\[
I_{\nu}(z)=\frac{(z/2)^{\nu}}{2\pi i}
\int_{c-i\infty}^{c+i\infty}\!\!t^{-\nu -1}
\exp\left(t+\frac{z^2}{4t} \right)dt
\]
(see Watson \cite{Watson:BF}) with $\nu=1$ and
\[
\frac{z}{2}=\left[ \frac{4\pi^2 r_ks_k}{k^2rs} \left(\delta_k-m
\right)\left(n-R\right) \right]^{1/2}
\]
to get, after some simplification,
\begin{multline}\label{prseq555}
p_{r,s}(n)=\sum_{k=1}^{\infty}\sum_{0\leq m <\delta_k} \frac{2\pi
A_{k,m}(n)}{k}\sqrt{\frac{r_ks_k(\delta_k-m)}{rs(n-R)}} \\
 \times
 I_1\left(\frac{4\pi}{k}\sqrt{\frac{r_ks_k}{rs}\left(\delta_k-m
\right)\left(n-R\right)}\right).
\end{multline}
The proof of Theorem \ref{t2} is now complete.

\section{Convergence behaviour}\label{cb}

Obtaining a bound for the error in using the $N$th partial sum in
Rademacher's series to estimate $p(n)$ is a difficult problem, and
 bounding the error in using the $N$th partial sum
of the series at \eqref{prseq555} to estimate $p_{r,s}(n)$ is likely to
be at least as difficult. We do not attempt  an analysis of this
problem in the present paper. However, we do examine a particular
numerical example, to get a feel for the speed and the nature of the
convergence.

\begin{table}[ht]  \begin{center}\label{Ta:t1}    \begin{tabular}{| c | c |c| }
       \hline
        $N$   &   $S_N$ &  $p_{14,15}(500)-S_N$ \\
       \hline  &    &    \\
 1       &310093947025049932429.8505 &$-2.374319315\times10^7$ \\
2       &310093947025073675628.9283  &5.9283       \\
3   &310093947025073675414.3591 &$-208.6409$   \\
4   &310093947025073675623.3258 &0.3258        \\
5    &310093947025073675623.3258 &0.3258  \\
6   &310093947025073675623.3723 &0.3723       \\
7    &310093947025073675623.3723 &0.3723  \\
8     &310093947025073675623.3723 &0.3723        \\
   9    &310093947025073675623.2793 &0.2793  \\
    10   &310093947025073675623.2793 &0.2793  \\
  11    &310093947025073675623.4447 &0.4447  \\
    \hline
     \end{tabular}
     \caption{The fast initial convergence of the
series for $p_{14,15}(500)$.}
       \end{center}
\end{table}

In the case examined ($r=14$, $s=15$, $n=500$), the convergence of
the series is initially very fast, while it seems that once the
partial sums of the series get to within 1.0 of the correct value,
that convergence then proceeds much more slowly, with (for $k \geq r
s$) the greatest contributions to the sum of the series coming from
those terms with $k \equiv 0 \pmod{rs}$, and with the contributions
from the terms for the other $k$ being negligible in comparison.
However, it is possible that the convergence behaviour may be
different, if $r$ and $s$ have a different number of prime factors
than in the example.

As an illustration of the convergence  behaviour,  we consider the
convergence of the sum of the series to
$$p_{14,15}(500)=310,093,947,025,073,675,623,$$ by examining the
difference  $p_{14,15}(500)-S_N$, where $S_N$ is the $N$th partial
sum of the series. We tabulate the values for $1\leq N \leq 11$ in
Table \ref{Ta:t1} to show the very fast convergence initially.

Note
that the terms in the series corresponding to $k=5$, $k=7$ and
$k=10$ are zero, since $\delta_5, \delta_7, \delta_{10}<0$, so that
each of the inner sums over $m$ are empty, and thus contribute zero
to the value of the series (the term in the series corresponding to
$k=8$ is also zero, but this is because the terms in the inner sum
over $m$ add to zero).

We next plot (Figure \ref{fig3}) the values for $1\leq N\leq 750$
(the large initial values lie outside the range of the plot) to show
how the terms in the sum corresponding to $k=210, 420, 630, \dots$
(multiples of $r\times s = 14\times 15$) contribute much more to the
value of tail of the series than values of $k \not \equiv 0
\pmod{rs}$.

\begin{figure}[htbp]
\centering \epsfig{file=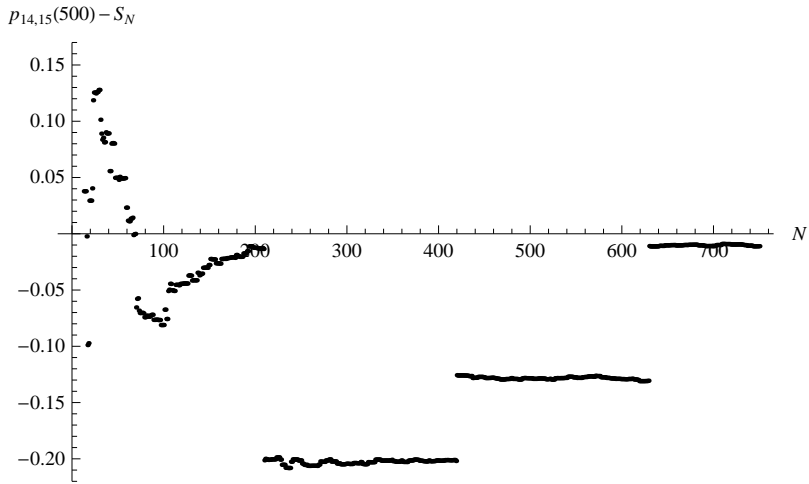, width= 300pt} \caption{The
convergence of the series to $p_{14,15}(500)$. The plot shows the
values of $p_{14,15}(500)-S_N$, where $S_N$ is the $N$th partial
sum of the series. Note the jumps in the values of the partial sums
for $k=210, 420$ and $630$.} \label{fig3}
\end{figure}

\begin{figure}[htbp]
\centering \epsfig{file=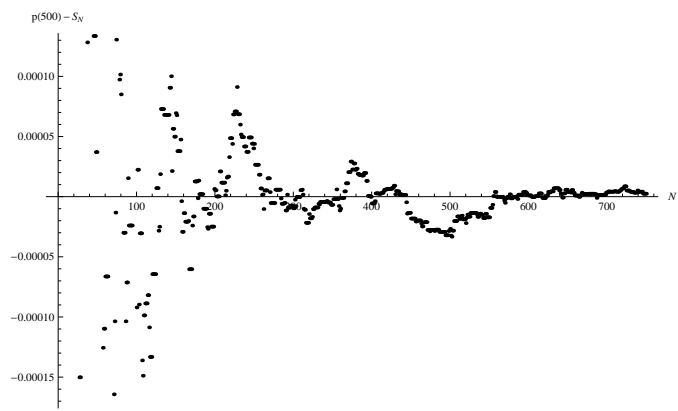, width= 250pt} \caption{The
convergence of Rademacher's series to $p(500)$. Note the more
erratic convergence behaviour, compared that of the series for
$p_{14,15}(500)$.} \label{fig3a}
\end{figure}

We remark that this apparent step-like convergence behaviour of the
series for $p_{r,s}(n)$ is in contrast to the apparent convergence
behaviour of the Rademacher series for $p(n)$, which is more erratic.
Figure \ref{fig3a} is a plot of the difference $p(500)-S_N$, $1\leq
N \leq 750$, where  $S_N$ is the $N$th partial sum of Rademacher's
series.

\begin{figure}[htbp]
\centering \epsfig{file=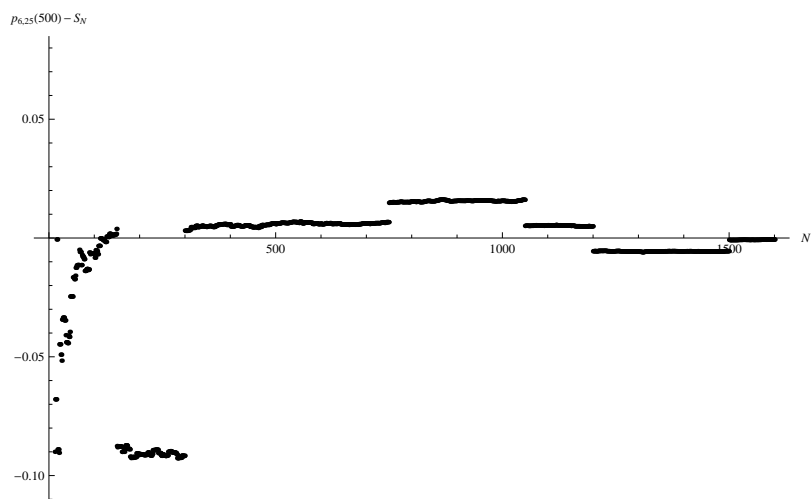, width= 300pt}
\caption{The  series corresponding to $p_{6,25}(500)$ appears to
converge to $p_{6,25}(500)$, despite the fact that 25 is a square. }
\label{fig3b}
\end{figure}

We conclude by remarking that experimental evidence suggests that
the requirement that $r$ and $s$ be square-free may be dropped,
although it is not possible to employ the arguments used to get the
Kloosterman sum estimates in this case. For example, seven terms of
the series for $$p_{6,25}(500)=42,305,606,435,448,427,065$$ appear to
be sufficient to get within $0.5$ of $p_{6,25}(500)$. Figure
\ref{fig3b} is a plot of difference $p_{6,25}(500)-S_N$, $1\leq N
\leq 1600$, where once again $S_N$ is the $N$th partial sum of the
series in Theorem \ref{t2}.

 Note that the convergence of the series
for $p_{6,25}(500)$ exhibits the same step-like behaviour seen above
 in the convergence of the series for $p_{14,15}(500)$,
 with the steps this time being multiples of $150=6\times 25$.


\begin{thebibliography}{100}

\bibitem{Apostol:MF}
T.~M. Apostol, \emph{{M}odular functions and {D}irichlet series in number
  theory}, Springer, 1990.

\bibitem{Bak:DI}
R.~C. Baker, \emph{{D}iophantine inequalities}, Clarendon Press, Oxford, 1986.

\bibitem{BO11}
K. Bringmann and K. Ono, \emph{Coefficients of harmonic Maass
forms}, Proceedings of the 2008 University of Florida Conference on
Partitions, q-series, and Modular Forms, Developments in Mathematics
series, Springer, to appear.

\bibitem{Hab:41}
M.~Haberzetle, \emph{On some partition functions}, Amer. J. Math.
\textbf{63} (1941), 589–-599.

\bibitem{Hagis:62}
P.~Hagis, \emph{A problem on partitions with a prime modulus $p \geq 3$},
  Trans.\ Amer.\ Math.\ Soc.\ \textbf{102} (1962), 30--62.

\bibitem{Hagis:63}
P.~Hagis, \emph{Partitions into odd summands}, Amer. J. Math.
\textbf{85} (1963), 213–-222.

\bibitem{Hagis:64}
P.~Hagis, \emph{On a class of partitions with distinct summands},
Trans. Amer. Math. Soc. \textbf{112} (1964), 401–-415.

\bibitem{Hagis:64b}
P.~Hagis, \emph{Partitions into odd and unequal parts}, Amer. J.
Math. \textbf{86} (1964), 317–-324.

\bibitem{Hagis:65}
P.~Hagis, \emph{ On the partitions of an integer into distinct odd
summands}, Amer. J. Math. \textbf{87} (1965), 867–-873.

\bibitem{Hagis:66}
P.~Hagis, \emph{Some theorems concerning partitions into odd
summands}, Amer. J. Math. \textbf{88} (1966), 664–-681.

\bibitem{Hagis:71}
P.~Hagis, \emph{Partitions with a restriction on the multiplicity of the
  summands}, Trans.\ Amer.\ Math.\ Soc. \textbf{155} (1971), 375--384.

\bibitem{Hagis:71b}
P.~Hagis, \emph{Partitions into unequal parts satisfying certain
congruence conditions}, J. Number Theory \textbf{3} (1971), 115-–123.

\bibitem{HR:18}
G.~H. Hardy and S.~Ramanujan, \emph{Asymptotic formulae in combinatory
  analysis}, Proc.\ London Math.\ Soc. (2)\ \textbf{17} (1918), 75--115.

\bibitem{H42}
L. K.~Hua \emph{On the number of partitions of a number into unequal
parts}, Trans. Amer. Math. Soc. \textbf{51} (1942), 194–-201.

\bibitem{I:59}
S.~Iseki, \emph{A partition function with some congruence condition},
Amer. J. Math. \textbf{81} (1959), 939–-961.

\bibitem{I:60}
S.~Iseki, \emph{On some partition functions}, J. Math. Soc. Japan
\textbf{12} (1960), 81–-88.

\bibitem{I:61}
S.~Iseki, \emph{Partitions in certain arithmetic progressions}, Amer.
J. Math. \textbf{83 }(1961), 243–-264.

\bibitem{IK:ANT}
H.~Iwaniec and E.~Kowalski, \emph{Analytic number theory}, American
  Mathematical Society, 2004.

\bibitem{Lehner:41}
J.~Lehner, \emph{A partition function connected with the modulus five}, Duke
  Math. J. \textbf{8} (1941), 631--655.

\bibitem{L45}
J.~Livingood, \emph{A partition function with the prime modulus
$P>3$}, Amer. J. Math. \textbf{67} (1945), 194–-208.

\bibitem{Niven:40}
I.~Niven, \emph{On a certain partition function}, Amer. J. Math.
\textbf{62} (1940), 353--364.

\bibitem{PS01}
A. V. M. Prasad and  V. V. S. Sastri, \emph{H.R.R. series for
certain F-partitions using Ford circles}, Ranchi Univ. Math. J.
\textbf{31} (2000), 51–-63.

\bibitem{Rad:37}
H.~Rademacher, \emph{On the partition function $p(n)$}, Proc. London
Math. Soc. (2) \textbf{43} (1937), 241–-254.

\bibitem{S72}
V. V. S. Sastri, \emph{Partitions with congruence conditions}, J. Indian
Math. Soc. (N.S.) \textbf{36} (1972), 177–-194.

\bibitem{SV82}
V. V. S. Sastri and S.  Vangipuram,
 \emph{A problem on partitions with
congruence conditions}, Indian J. Math. \textbf{24} (1982), no. 1-3,
165–-174.

\bibitem{S10a}
A.~Sills, \emph{Towards an automation of the circle method}, Gems in
experimental mathematics, 321–-338, Contemp. Math., \textbf{517},
Amer. Math. Soc., Providence, RI, 2010.

\bibitem{S10b}
A.~Sills, \emph{Rademacher-type formulas for restricted partition and
overpartition functions}, Ramanujan J. \textbf{23} (2010), no. 1-3,
253–-264.

\bibitem{S10c}
A.~Sills, \emph{A Rademacher type formula for partitions and
overpartitions}, Int. J. Math. Math. Sci. 2010, Art. ID 630458, 21
pp.

\bibitem{Watson:BF}
G.~B. Watson, \emph{A treatise on the theory of {B}essel functions},
2nd ed., Cambridge University Press, 1944.


\end{thebibliography}
\end{document}